\documentclass[11pt]{article}

\usepackage[a4paper,margin=1in]{geometry}
\usepackage[T1]{fontenc}
\usepackage{lmodern}
\usepackage{microtype}
\usepackage{amsmath,amssymb,amsthm}
\usepackage{xcolor}
\usepackage[hypertexnames=false]{hyperref}

\numberwithin{equation}{section}

\newtheorem{theorem}{Theorem}[section]
\newtheorem{proposition}[theorem]{Proposition}
\newtheorem{lemma}[theorem]{Lemma}
\newtheorem{corollary}[theorem]{Corollary}
\newtheorem{conjecture}[theorem]{Conjecture}
\theoremstyle{remark}
\newtheorem{remark}[theorem]{Remark}

\newcommand{\R}{\mathbb R}
\newcommand{\Z}{\mathbb Z}
\newcommand{\norm}[1]{\left\lVert #1\right\rVert}
\newcommand{\abs}[1]{\left\lvert #1\right\rvert}
\newcommand{\transpose}{\mathsf T}

\hypersetup{
 colorlinks=true,
 linkcolor=blue!55!black,
 citecolor=blue!55!black,
 urlcolor=blue!55!black,
 pdftitle={Lp-Integrability of Radon--Nikodym Densities Between Harmonic Energy Measures on the Sierpinski Gasket},
 pdfauthor={Konstantinos Tsougkas},
 pdfsubject={Energy measures on the Sierpinski gasket},
 pdfkeywords={Sierpinski gasket, energy measure, Radon--Nikodym derivative, harmonic function, density-ratio power sum}
 }

\title{$L^p$-Integrability of Radon--Nikodym Densities Between\\
Harmonic Energy Measures\\
 on the Sierpi\'nski Gasket}
\author{Konstantinos Tsougkas\\
\small School of Engineering Science, University of Skövde\\
\small Box 408, 541 28 Skövde, Sweden\\
\small \texttt{konstantinos.tsougkas@his.se}}
\date{July 24, 2026}

\begin{document}

\maketitle
\begin{abstract}
It is known that the energy measures of any two nonconstant harmonic
functions on the standard Sierpi\'nski gasket are mutually absolutely
continuous.  Strichartz and Tse reported numerical evidence for
$L^p$-integrability of the corresponding Radon--Nikodym densities in the
range
\[
 1<p<\frac{\log 15}{\log 9}.
\]
For arbitrary ordered pairs of nonconstant harmonic functions, we prove
uniform boundedness of the associated density-ratio power sums, and hence
$L^p$-integrability, in the subinterval
\[
 1<p<\frac{\log(35/3)}{\log 9}.
\]
When the denominator harmonic direction is represented by the boundary
values $(0,-1,1)$, we prove boundedness throughout the full conjectured
interval.  
\end{abstract}
\medskip

\section{Introduction}

\subsection{Background and the Strichartz--Tse conjecture}

The field of analysis on fractals, developed by Kigami \cite{Kigami} and
Strichartz \cite{StrichartzBook}, equips certain self-similar sets with
resistance forms that play the role of the classical Dirichlet energy.  The
standard Sierpi\'nski gasket, denoted by $K$, is generated by the three
contractions
\[
 F_i(x)=\frac12(x-q_i)+q_i,
 \qquad i=0,1,2,
\]
where $q_0,q_1,q_2$ are the vertices of an equilateral triangle.  We write
$V_0=\{q_0,q_1,q_2\}$ for the boundary of $K$.  A \emph{word} of length
$n\ge1$ is the string $w=i_1\cdots i_n$ with $i_j\in\{0,1,2\}$ and we set
$\abs{w}=n$ and $F_w=F_{i_1}\circ\cdots\circ F_{i_n}$
so that the \emph{cells} $F_w(K)$ with $\abs{w}=n$ form the level-$n$ cell
decomposition of $K$.  The standard resistance form on $K$ is denoted by
$\mathcal E(\cdot,\cdot)$, and we write $\mathcal E(f)=\mathcal E(f,f)$.
The energy satisfies the self-similar identity
\[
 \mathcal E(f)=\frac53\sum_{i=0}^2\mathcal E(f\circ F_i).
\]

Associated with each function $f$ in the energy domain is its energy
measure $\mu_f$.  On the Sierpi\'nski gasket, these measures are singular
with respect to the standard self-similar measure
\cite{HinoNakahara2006,Kusuoka}.  Kusuoka introduced a distinguished
energy-dominant measure, now called the Kusuoka measure, with respect to
which every energy measure is absolutely continuous; see
\cite{Kigami2008,Kusuoka}.  More is true for harmonic functions: Hino
\cite{Hino2010} proved that the energy measures of any two
nonconstant harmonic functions on $K$ are mutually absolutely continuous.  Further
properties of harmonic energy measures are studied in
\cite{BellHoStrichartz2014,Hino2016}.  For the broader family of level-$k$
Sierpi\'nski gaskets $SG_k$ (with $SG_2=K$), related results on the Kusuoka
measure, the energy Laplacian, and non-degeneracy of the harmonic structure
appear in \cite{ObergTsougkas2019,Tsougkas2019}.  The present paper asks the
following quantitative question: how integrable is this density?

For a nonconstant harmonic function $h$, let $\mu_h$ denote its energy
measure and let
\[
 \nu_h=\frac{\mu_h}{\mu_h(K)}
\]
be its normalized version.  Given an ordered pair of
nonconstant harmonic functions $h_1,h_2$, we can define the level-$n$
\emph{density-ratio power sum}
\begin{equation}\label{eq:Smp}
 S_{h_1,h_2}(n,p)
 =\sum_{\abs{w}=n}
   \nu_{h_1}(F_w(K))^p\nu_{h_2}(F_w(K))^{1-p}.
\end{equation}
Proposition~\ref{prop:martingale-criterion} below shows that uniform
boundedness of $S_{h_1,h_2}(n,p)$ in $n$ implies, for $p>1$, that
\[
 \frac{d\nu_{h_1}}{d\nu_{h_2}}\in L^p(\nu_{h_2}).
\]
Strichartz and Tse \cite{StrichartzTse} proposed the following conjecture,
motivated by their numerical experiments, which is equivalent to the uniform boundedness of this sum in $n$.

\begin{conjecture}\label{conj:ST}
For nonconstant harmonic functions $h_1,h_2$ on
$K$, we have that $ \frac{d\nu_{h_1}}{d\nu_{h_2}}\in L^p(\nu_{h_2})$ for 
\[
 1<p<\frac{\log 15}{\log 9}.
\]
\end{conjecture}

Here and below, the \emph{harmonic direction} of a nonconstant harmonic
function is its equivalence class under transformations
$h\mapsto\alpha h+\beta$, where $\alpha,\beta\in\R$ and $\alpha\ne0$.

The exponent $p$ cannot be improved for all ordered pairs. Indeed for the ordered pair
whose harmonic directions are represented by the boundary values $(1,0,0)$
and $(0,-1,1)$, respectively, boundedness fails whenever
$p>\log 15/\log 9$.  Indeed, as noted in \cite{StrichartzTse} the single cell $F_0^n(K)$ satisfies
\[
 \nu_{h_1}(F_0^n(K))=\left(\frac35\right)^n,
 \qquad
 \nu_{h_2}(F_0^n(K))=\left(\frac1{15}\right)^n,
\]
so its contribution to \eqref{eq:Smp} is
\begin{equation}\label{eq:intro-zero-branch}
 \left(\frac{9^p}{15}\right)^n.
\end{equation}
Consequently, uniform boundedness fails for this ordered pair when
$p>\log 15/\log 9$.  At the endpoint the contribution
\eqref{eq:intro-zero-branch} equals $1$, so endpoint divergence for this pair
requires a finer argument which we supply in
Theorem~\ref{thm:sharp}.

\subsection{Main results and proof strategy}

We establish three main results concerning Conjecture~\ref{conj:ST}.

\begin{enumerate}
\item \emph{A universal subcritical range}
 (Theorem~\ref{thm:general}).  For every ordered pair of nonconstant
 harmonic functions $h_1,h_2$, the sequence $S_{h_1,h_2}(n,p)$ is bounded
 uniformly in $n$ for all
 \[
  1<p<\frac{\log(35/3)}{\log 9}.
 \]
\item \emph{The full conjectured range for a distinguished denominator
 harmonic direction}
 (Corollary~\ref{cor:boundary-denominator}).  If the harmonic direction of
 $h_2$ is represented by the boundary values $(0,-1,1)$, then, with no
 restriction on the nonconstant numerator $h_1$, boundedness holds for
 \[
  1 < p<\frac{\log 15}{\log 9}.
 \]
\item \emph{Sharpness for a particular ordered pair}
 (Theorem~\ref{thm:sharp}).  For the ordered pair of harmonic directions
 represented by the boundary values $(1,0,0)$ and $(0,-1,1)$, the threshold
 $p_*=\log 15/\log 9$ is exact.  At $p=p_*$, the density-ratio power sums
 grow at least linearly in $n$, while above $p_*$ they grow at least
 exponentially.
\end{enumerate}

To a nonconstant harmonic function
$h$ we associate a nonzero vector
$z_h\in\R^2$ such that
\[
 \nu_h(F_w(K))=
 \frac{\norm{K_wz_h}^2}{15^{\abs{w}}\norm{z_h}^2}
\]
for three explicit matrices $K_0,K_1,K_2$.
We then bound the successive increments of \eqref{eq:Smp} in terms of the
reciprocal sum
\[
 R_n(v)=\sum_{\abs{w}=n}\frac1{\norm{K_wv}^2},
\]
and obtain the universal bound
\[
 R_n(v)\le\frac1{\norm{v}^2}\left(\frac97\right)^n.
\]
When $v=e_2$, the lattice coding obtained from the integer conjugacy in
Lemma~\ref{lem:lattice} is injective on words of each fixed length.  The
resulting lattice-sum estimate gives $R_n(e_2)=O(n)$.  This polynomial bound
yields the full conjectured interval when the denominator harmonic direction
is represented by the boundary values $(0,-1,1)$.  Throughout, these
estimates control successive increments of the density-ratio power sums, and
we will also use the following elementary summability observation from Calculus.

\begin{lemma}\label{lem:increment-summability}
Let $(a_n)_{n\ge1}$ be a real sequence.
\begin{enumerate}
\item If
\[
 0\le a_{n+1}-a_n\le b_n
\]
for a nonnegative summable sequence $(b_n)_{n\ge1}$, then $(a_n)$ is
nondecreasing and bounded, and
\[
 \sup_{N\ge1}a_N\le a_1+\sum_{n=1}^{\infty}b_n.
\]
\item If
\[
 a_{n+1}-a_n\ge c
\]
for some constant $c>0$ and every $n\ge1$, then
\[
 a_N\ge a_1+c(N-1)
 \qquad(N\ge1).
\]
\end{enumerate}
\end{lemma}

\section{A two-dimensional formulation for harmonic functions}
\label{sec:harmonic-formulation}

\subsection{Boundary coordinates and cell-mass formulas}

Let $\mathcal E(\cdot,\cdot)$ be the standard resistance form and, as above,
write $\mathcal E(f)=\mathcal E(f,f)$.  For a nonconstant harmonic function
$h$, its unnormalized energy measure $\mu_h$ is characterized by
\begin{equation}\label{eq:energy-measure-definition}
 \int_{K}\varphi\,d\mu_h
 =\mathcal E(h,\varphi h)-\frac12\mathcal E(\varphi,h^2),
\end{equation}
for test functions $\varphi$ in the energy domain.  This fixes the convention
used here; see \cite{AzzamHallStrichartz}.  Taking
$\varphi=1$ in \eqref{eq:energy-measure-definition} gives
$\mu_h(K)=\mathcal E(h)$ and it also holds that
\[
 \mu_h(K)
 =\bigl(h(q_0)-h(q_1)\bigr)^2
  +\bigl(h(q_1)-h(q_2)\bigr)^2
 +\bigl(h(q_2)-h(q_0)\bigr)^2.
\]
The standard
local-energy characterization on a closed cell is
\begin{equation}\label{eq:local-energy-cell}
 \mu_h(F_w(K))=\mathcal E_{F_w(K)}(h)
 =\left(\frac53\right)^{\abs{w}}\mathcal E(h\circ F_w);
\end{equation}
see \cite{AzzamHallStrichartz}.  We use the normalized energy
measure
\[
 \nu_h=\frac{\mu_h}{\mu_h(K)},
 \qquad
 \nu_h(K)=1.
\]
See also \cite{Kigami} for the standard resistance-form and
harmonic-extension background used here.

The corresponding unnormalized sum satisfies
\[
 \sum_{\abs{w}=n}\mu_{h_1}(F_w(K))^p\mu_{h_2}(F_w(K))^{1-p}
 =
 \mu_{h_1}(K)^p\mu_{h_2}(K)^{1-p}S_{h_1,h_2}(n,p)
\]
thus normalization changes the sequence only by a positive factor independent
of $n$ and does not affect boundedness, monotonicity, or the set of exponents
for which the sequence is bounded.

Denote the three scalar boundary values of $h$ by
\[
 a=h(q_0),
 \qquad
 b=h(q_1),
 \qquad
 c=h(q_2).
\]
Define the scalars $x,y$ and the vector $z_h$ by
\begin{equation}\label{eq:harmonic-coordinates}
 x=a-\frac{b+c}{2},
 \qquad
 y=\frac{\sqrt3}{2}(c-b),
 \qquad
 z_h=\begin{pmatrix}x\\y\end{pmatrix}\in\R^2.
\end{equation}
Throughout, $\norm{\cdot}$ denotes the Euclidean norm on $\R^2$.
Adding the same constant to $a,b,c$ leaves both $x$ and $y$
unchanged and it can be easily seen that $h$ is
nonconstant exactly when $z_h\ne0$.

Accordingly, $x$ and $y$ determine the boundary-value triple up to the
addition of a common constant.  Explicitly, solving
\eqref{eq:harmonic-coordinates} gives
\[
 (a,b,c)=
 \left(
  \frac{2x}{3},
  -\frac{x}{3}-\frac{y}{\sqrt3},
  -\frac{x}{3}+\frac{y}{\sqrt3}
 \right)
 +t(1,1,1),
 \qquad t\in\R.
\]
In particular, the boundary differences are recovered directly
from $x$ and $y$:
\[
 a-b=x+\frac{y}{\sqrt3},
 \qquad
 b-c=-\frac{2y}{\sqrt3},
 \qquad
 c-a=\frac{y}{\sqrt3}-x.
\]
Substitution into the boundary energy gives
\begin{align}
\mu_h(K)
&=\left(x+\frac{y}{\sqrt3}\right)^2
 +\left(-\frac{2y}{\sqrt3}\right)^2
 +\left(\frac{y}{\sqrt3}-x\right)^2\notag\\
&=2x^2+2y^2
 =2\norm{z_h}^2.
\label{eq:total-energy-coordinate}
\end{align}

Define
\begin{equation*}
K_0=
\begin{pmatrix}3&0\\0&1\end{pmatrix},
\qquad
K_1=\frac12
\begin{pmatrix}3&\sqrt3\\ \sqrt3&5\end{pmatrix},
\qquad
K_2=\frac12
\begin{pmatrix}3&-\sqrt3\\ -\sqrt3&5\end{pmatrix}.
\end{equation*}
Each $K_i$ is symmetric, has trace $4$ and determinant $3$, and therefore has
eigenvalues $1$ and $3$.  Moreover,
\begin{equation}\label{eq:K-basic}
 \norm{z}\le\norm{K_i z}\le3\norm{z}
 \quad(z\in\R^2,\ i=0,1,2),
 \qquad
 \sum_{i=0}^2K_i^\transpose K_i=15I.
\end{equation}
We write the standard basis vectors of $\R^2$ as
\[
 e_1=\begin{pmatrix}1\\0\end{pmatrix},
 \qquad
 e_2=\begin{pmatrix}0\\1\end{pmatrix}.
\]
For a word $w=i_1\cdots i_n$, define
\begin{equation}\label{eq:word-matrix}
 K_w=K_{i_n}\cdots K_{i_1}.
\end{equation}
With this product convention, $K_{wi}=K_iK_w$.  A standard
three-dimensional matrix formalism for computing harmonic energy measures
appears in \cite{AzzamHallStrichartz}.  Here we instead develop
a two-dimensional formulation adapted to the estimates in this paper.

\begin{lemma}
Let $h$ be a nonconstant harmonic function.  For each $i\in\{0,1,2\}$,
\begin{equation}\label{eq:coordinate-recursion}
 z_{h\circ F_i}=\frac15K_i z_h.
\end{equation}
More generally, for every word $w$,
\begin{equation}\label{eq:word-coordinate-recursion}
 z_{h\circ F_w}=\frac1{5^{\abs{w}}}K_wz_h.
\end{equation}
Consequently,
\begin{equation}\label{eq:unnormalized-cell-formula}
 \mu_h(F_w(K))=\frac2{15^{\abs{w}}}\norm{K_wz_h}^2,
\end{equation}
and
\begin{equation}\label{eq:exact-cell-formula}
 \nu_h(F_w(K))=
 \frac{\norm{K_wz_h}^2}{15^{\abs{w}}\norm{z_h}^2}.
\end{equation}
\end{lemma}

\begin{proof}
The harmonic extension rule gives
\begin{align*}
h(F_0q_1)=h(F_1q_0)&=\frac{2a+2b+c}{5},\\
h(F_0q_2)=h(F_2q_0)&=\frac{2a+b+2c}{5},\\
h(F_1q_2)=h(F_2q_1)&=\frac{a+2b+2c}{5}.
\end{align*}
Therefore the boundary triples of the three restricted functions are
\begin{align*}
h\circ F_0:\quad&
\left(a,\frac{2a+2b+c}{5},\frac{2a+b+2c}{5}\right),\\
h\circ F_1:\quad&
\left(\frac{2a+2b+c}{5},b,\frac{a+2b+2c}{5}\right),\\
h\circ F_2:\quad&
\left(\frac{2a+b+2c}{5},\frac{a+2b+2c}{5},c\right).
\end{align*}

We illustrate the coordinate calculation for $h \circ F_0$.  Its
first coordinate is
\begin{align*}
a-\frac12\left(\frac{2a+2b+c}{5}
                 +\frac{2a+b+2c}{5}\right)
&=\frac35\left(a-\frac{b+c}{2}\right)
 =\frac{3x}{5},
\end{align*}
while its second coordinate is
\[
 \frac{\sqrt3}{2}
 \left(\frac{2a+b+2c}{5}-\frac{2a+2b+c}{5}\right)
 =\frac{y}{5}.
\]
For $h\circ F_1$, the two coordinates simplify as follows:
\begin{align*}
\frac{2a+2b+c}{5}
-\frac12\left(b+\frac{a+2b+2c}{5}\right)
&=\frac{3(a-b)}{10}
 =\frac{3x+\sqrt3y}{10},\\
\frac{\sqrt3}{2}
\left(\frac{a+2b+2c}{5}-b\right)
&=\frac{\sqrt3(a-3b+2c)}{10}
 =\frac{\sqrt3x+5y}{10}.
\end{align*}
For $h\circ F_2$, they are
\begin{align*}
\frac{2a+b+2c}{5}
-\frac12\left(\frac{a+2b+2c}{5}+c\right)
&=\frac{3(a-c)}{10}
 =\frac{3x-\sqrt3y}{10},\\
\frac{\sqrt3}{2}
\left(c-\frac{a+2b+2c}{5}\right)
&=\frac{\sqrt3(3c-a-2b)}{10}
 =\frac{-\sqrt3x+5y}{10}.
\end{align*}
We have therefore obtained
\begin{align*}
z_{h\circ F_0}
&=\frac15\begin{pmatrix}3x\\y\end{pmatrix},\\
z_{h\circ F_1}
&=\frac1{10}
  \begin{pmatrix}3x+\sqrt3y\\ \sqrt3x+5y\end{pmatrix},\\
z_{h\circ F_2}
&=\frac1{10}
  \begin{pmatrix}3x-\sqrt3y\\ -\sqrt3x+5y\end{pmatrix}.
\end{align*}
These are exactly the three identities in \eqref{eq:coordinate-recursion}.
Iteration proves \eqref{eq:word-coordinate-recursion} with the product order
in \eqref{eq:word-matrix}.  By the local-energy identity
\eqref{eq:local-energy-cell},
\[
 \mu_h(F_w(K))
 =\left(\frac53\right)^{\abs{w}}\mathcal E(h\circ F_w)
 =\left(\frac53\right)^{\abs{w}}\mu_{h\circ F_w}(K).
\]
Using \eqref{eq:total-energy-coordinate} and
\eqref{eq:word-coordinate-recursion}, we obtain
\begin{align*}
\mu_h(F_w(K))
&=\left(\frac53\right)^{\abs{w}}
  2\norm{z_{h\circ F_w}}^2\\
&=\left(\frac53\right)^{\abs{w}}
  \frac2{5^{2\abs{w}}}\norm{K_wz_h}^2\\
&=\frac2{15^{\abs{w}}}\norm{K_wz_h}^2.
\end{align*}
This proves \eqref{eq:unnormalized-cell-formula}.  Dividing by
$\mu_h(K)=2\norm{z_h}^2$ proves \eqref{eq:exact-cell-formula}.
\end{proof}

We can also obtain the following already known standard estimate by using \eqref{eq:K-basic}. We get that
\begin{equation}\label{eq:cell-mass-bounds}
 15^{-\abs{w}}
 \le \nu_h(F_w(K))
 \le \left(\frac35\right)^{\abs{w}}
\end{equation}
for every nonconstant harmonic function $h$ and every word $w$.  Indeed,
successive application of \eqref{eq:K-basic} gives
\[
 \norm{z_h}\le\norm{K_wz_h}
 \le3^{\abs{w}}\norm{z_h},
\]
and the assertion follows from \eqref{eq:exact-cell-formula}.

Adding a constant to $h$ leaves $z_h$ unchanged, and multiplying $h$ by a
nonzero scalar rescales $z_h$ but does not change the normalized energy
measure.  

The following Proposition can be obtained and we omit the proof as it follows from the discussion in \cite{StrichartzTse}.
\begin{proposition}
\label{prop:martingale-criterion}
Let $h_1,h_2$ be nonconstant harmonic functions and let $p>1$.  Then
$S_{h_1,h_2}(n,p)$ is nondecreasing in $n$, and the following statements are
equivalent:
\begin{enumerate}
\item $\sup_{n\ge1}S_{h_1,h_2}(n,p)<\infty$;
\item $\nu_{h_1}\ll\nu_{h_2}$ and
$\frac{d\nu_{h_1}}{d\nu_{h_2}}\in L^p(\nu_{h_2})$.
\end{enumerate}
\end{proposition}
The goal of this paper is to study the sum in the first statement of the above proposition.

\section{The common increment estimate}\label{sec:increment}

The following quantity will play a crucial role in our results:
\begin{equation*}
 R_n(v)=\sum_{\abs{w}=n}\frac1{\norm{K_wv}^2},
 \qquad n\ge1,\quad v\in\R^2\setminus\{0\}.
\end{equation*}
This sum is well defined and finite because there are $3^n$ words of length
$n$ and every $K_w$ is invertible.  We first isolate an estimate comparing
the two refinements inside a single cell.  In the following, $\det(u,v)$
denotes the determinant of the $2\times2$ matrix whose columns are the
vectors $u$ and $v$.

\begin{lemma}\label{lem:one-cell-comparison}
Let $u,v\in\R^2\setminus\{0\}$, let $p>1$, and define the scalars
\[
 P_i=\frac{\norm{K_i u}^2}{15\norm{u}^2},
 \qquad
 Q_i=\frac{\norm{K_i v}^2}{15\norm{v}^2},
 \qquad i=0,1,2.
\]
Then
\[
 \sum_{i=0}^2P_i=\sum_{i=0}^2Q_i=1,
 \qquad
 \frac1{15}\le P_i,Q_i\le\frac35,
\]
and there exists a constant $C_p>0$, depending only on $p$, such that
\begin{equation}\label{eq:one-cell-comparison}
 0\le \sum_{i=0}^2P_i^pQ_i^{1-p}-1
 \le C_p\frac{\det(u,v)^2}{\norm{u}^2\norm{v}^2}.
\end{equation}
\end{lemma}

\begin{proof}
Using \eqref{eq:K-basic}, we first obtain
\begin{align*}
 \sum_{i=0}^2P_i
 &=\frac{u^\transpose
   \left(\sum_{i=0}^2K_i^\transpose K_i\right)u}
  {15\norm{u}^2}
 =\frac{15\norm{u}^2}{15\norm{u}^2}=1.
\end{align*}
The same calculation gives $\sum_iQ_i=1$.  The bounds in
\eqref{eq:K-basic} give
\[
 \norm{u}^2\le\norm{K_i u}^2\le9\norm{u}^2,
 \qquad
 \norm{v}^2\le\norm{K_i v}^2\le9\norm{v}^2.
\]
Dividing by the corresponding denominators proves
\[
 \frac1{15}\le P_i,Q_i\le\frac35.
\]

The lower bound in \eqref{eq:one-cell-comparison} follows from weighted
Jensen's inequality.  Since the positive numbers $Q_i$ sum to $1$ and
$t\mapsto t^p$ is strictly convex for $p>1$,
\begin{align*}
 \sum_iP_i^pQ_i^{1-p}
 &=\sum_iQ_i\left(\frac{P_i}{Q_i}\right)^p\\
 &\ge
 \left(\sum_iQ_i\frac{P_i}{Q_i}\right)^p
 =\left(\sum_iP_i\right)^p=1.
\end{align*}
Equality holds exactly when all three ratios $P_i/Q_i$ are equal, which,
because both quantities sum to $1$, is equivalent to $P_i=Q_i$ for every
$i$.

For the upper bound, define $f_q(t)=t^pq^{1-p}$.  Then
\[
 f_q(q)=q,\qquad f_q'(q)=p,\qquad
 f_q''(t)=p(p-1)t^{p-2}q^{1-p}.
\]
For each fixed $p>1$, the second derivative is bounded uniformly for
$t,q\in[1/15,3/5]$.  Taking $t=P_i$ and $q=Q_i$, a second-order Taylor
expansion at $t=q$ therefore gives
\[
 P_i^pQ_i^{1-p}
 \le Q_i+p(P_i-Q_i)+C_p(P_i-Q_i)^2.
\]
After summing over $i$, the linear terms cancel because
\[
 \sum_i(P_i-Q_i)=0.
\]
Together with the Jensen lower bound, this proves
\begin{equation}\label{eq:probability-gap-bound}
 0\le\sum_iP_i^pQ_i^{1-p}-1
 \le C_p\sum_i(P_i-Q_i)^2.
\end{equation}

Choose polar angles $\varphi,\psi\in\R$ such that
\[
 \frac{u}{\norm{u}}=(\cos\varphi,\sin\varphi)^\transpose,
 \qquad
 \frac{v}{\norm{v}}=(\cos\psi,\sin\psi)^\transpose,
\]
and let $\theta\in[0,\pi/2]$ be the smaller angle between the lines spanned by
$u$ and $v$.  Then
\[
 \sin\theta=\abs{\sin(\varphi-\psi)}.
\]
For each $i$, choose an angle $\gamma_i\in\R$ such that
\[
 v_i^+=\bigl(\cos\gamma_i,\sin\gamma_i\bigr)^\transpose
\]
is a unit eigenvector of $K_i$ with eigenvalue $3$.  Because $K_i$ is
symmetric and its eigenvalues are distinct, the perpendicular unit vector
\[
 v_i^-=\bigl(-\sin\gamma_i,\cos\gamma_i\bigr)^\transpose
\]
is an eigenvector with eigenvalue $1$.  The angle-addition formulas give the
orthogonal decomposition
\[
 \frac{u}{\norm{u}}
 =\cos(\varphi-\gamma_i)v_i^+
  +\sin(\varphi-\gamma_i)v_i^-.
\]
Applying $K_i$ gives
\[
 \frac{K_i u}{\norm{u}}
 =3\cos(\varphi-\gamma_i)v_i^+
  +\sin(\varphi-\gamma_i)v_i^-.
\]
It follows that
\begin{align*}
 \frac{\norm{K_i u}^2}{\norm{u}^2}
 &=9\cos^2(\varphi-\gamma_i)
   +\sin^2(\varphi-\gamma_i)\\
 &=1+8\cos^2(\varphi-\gamma_i).
\end{align*}
The analogous formula for $v$ contains $\psi$ in place of $\varphi$.  Using
\[
 \cos^2\xi-\cos^2\eta=-\sin(\xi+\eta)\sin(\xi-\eta)
\]
with $\xi=\varphi-\gamma_i$ and $\eta=\psi-\gamma_i$, we obtain
\begin{align*}
 \abs{P_i-Q_i}
 &=\frac8{15}
   \abs{\sin(\varphi+\psi-2\gamma_i)}
   \abs{\sin(\varphi-\psi)}\\
 &\le\frac8{15}\sin\theta.
\end{align*}
Squaring and summing over $i$ gives
\[
 \sum_{i=0}^2(P_i-Q_i)^2
 \le\frac{64}{75}\sin^2\theta.
\]
Finally, the planar area formula gives
\[
 \sin^2\theta=\frac{\det(u,v)^2}{\norm{u}^2\norm{v}^2}.
\]
Combining the last two displays with \eqref{eq:probability-gap-bound} and
absorbing the fixed factor $64/75$ into $C_p$ proves the upper bound in
\eqref{eq:one-cell-comparison}.
\end{proof}

We now apply Lemma~\ref{lem:one-cell-comparison} simultaneously to all cells
at a fixed level.

\begin{proposition}\label{prop:master-increment}
Let $h_1,h_2$ be nonconstant harmonic functions and let $p>1$.  Then
$S_{h_1,h_2}(n,p)$ is nondecreasing for $n\ge1$, and
\begin{equation}\label{eq:master-increment}
0\le S_{h_1,h_2}(n+1,p)-S_{h_1,h_2}(n,p)
\le C_{p,h_1,h_2}
\left(\frac{9^p}{15}\right)^nR_n(z_{h_2})
\end{equation}
for every $n\ge1$, where $C_{p,h_1,h_2}>0$ depends on $p,h_1,h_2$ but is
independent of $n$.
\end{proposition}

\begin{proof}
Fix $n\ge1$.  For every word $w$ of length $n$, define the refinement
proportions
\[
 P_i=\frac{\nu_{h_1}(F_{wi}(K))}{\nu_{h_1}(F_w(K))}
 =\frac{\norm{K_iK_wz_{h_1}}^2}{15\norm{K_wz_{h_1}}^2},
 \qquad
 Q_i=\frac{\nu_{h_2}(F_{wi}(K))}{\nu_{h_2}(F_w(K))}
 =\frac{\norm{K_iK_wz_{h_2}}^2}{15\norm{K_wz_{h_2}}^2},
 \qquad i=0,1,2.
\]
The ratios are well defined by \eqref{eq:cell-mass-bounds}, and
Lemma~\ref{lem:one-cell-comparison} applies to
$K_wz_{h_1}$ and $K_wz_{h_2}$.  Since every word of length $n+1$ is uniquely
of the form $wi$ and
\[
 \nu_{h_1}(F_{wi}(K))=\nu_{h_1}(F_w(K))P_i,
 \qquad
 \nu_{h_2}(F_{wi}(K))=\nu_{h_2}(F_w(K))Q_i,
\]
summing over $w$ and $i$ and subtracting the level-$n$ sum gives
\begin{align}
 &S_{h_1,h_2}(n+1,p)-S_{h_1,h_2}(n,p)\notag\\
 &\qquad=
 \sum_{\abs{w}=n}\nu_{h_1}(F_w(K))^p\nu_{h_2}(F_w(K))^{1-p}
 \left[
  \sum_{i=0}^2P_i^pQ_i^{1-p}-1
 \right].
 \label{eq:exact-increment}
\end{align}
The bracket is nonnegative by Lemma~\ref{lem:one-cell-comparison}, and every
factor outside it is positive.  Thus the increment is nonnegative, which
proves monotonicity.

We now use the upper bound from the same lemma.  Since every $K_i$ has
determinant $3$,
\begin{align*}
 \det(K_wz_{h_1},K_wz_{h_2})
 &=\det(K_w)\det(z_{h_1},z_{h_2}),\\
 \det(K_w)&=3^n.
\end{align*}
Consequently,
\begin{align*}
 0&\le\sum_iP_i^pQ_i^{1-p}-1\\
 &\le C_p
 \frac{9^n\det(z_{h_1},z_{h_2})^2}
 {\norm{K_wz_{h_1}}^2\norm{K_wz_{h_2}}^2}\\
 &\le C_{p,h_1,h_2}
 \frac{9^n}{\norm{K_wz_{h_1}}^2\norm{K_wz_{h_2}}^2},
\end{align*}
where the fixed determinant has been absorbed into the last constant.

By \eqref{eq:exact-cell-formula},
\[
 \nu_{h_1}(F_w(K))^p\nu_{h_2}(F_w(K))^{1-p}
 =\frac1{15^n}
 \frac{\norm{K_wz_{h_1}}^{2p}}{\norm{z_{h_1}}^{2p}}
 \frac{\norm{z_{h_2}}^{2p-2}}
 {\norm{K_wz_{h_2}}^{2p-2}}.
\]
Substituting this identity and the preceding estimate into
\eqref{eq:exact-increment} and noting that
$9^n/15^n=(3/5)^n$ gives
\begin{align*}
&S_{h_1,h_2}(n+1,p)-S_{h_1,h_2}(n,p)\\
&\quad\le
C_{p,h_1,h_2}\left(\frac35\right)^n
\sum_{\abs{w}=n}
\left(\frac{\norm{K_wz_{h_1}}^2}
{\norm{K_wz_{h_2}}^2}\right)^{p-1}
\frac1{\norm{K_wz_{h_2}}^2},
\end{align*}
where the fixed powers of $\norm{z_{h_1}}$ and $\norm{z_{h_2}}$ have also
been absorbed into the constant.

Applying the bounds in \eqref{eq:K-basic} successively to
the $n$ factors in $K_w$ gives
\[
 \norm{K_wz_{h_1}}\le3^n\norm{z_{h_1}},
 \qquad
 \norm{K_wz_{h_2}}\ge\norm{z_{h_2}},
\]
and thus
\[
 \frac{\norm{K_wz_{h_1}}^2}{\norm{K_wz_{h_2}}^2}
 \le9^n\frac{\norm{z_{h_1}}^2}{\norm{z_{h_2}}^2}.
\]
Since $p>1$,
\[
 \left(\frac{\norm{K_wz_{h_1}}^2}{\norm{K_wz_{h_2}}^2}\right)^{p-1}
 \le C_{p,h_1,h_2}9^{n(p-1)}.
\]
Moreover,
\[
 \left(\frac35\right)^n9^{n(p-1)}
 =\left(\frac{9^p}{15}\right)^n,
\]
and the remaining word sum is precisely $R_n(z_{h_2})$.  This proves
\eqref{eq:master-increment}.
\end{proof}

\section{Arbitrary harmonic functions: the universal interval}
\label{sec:general-pair}

To apply Proposition~\ref{prop:master-increment}, we first obtain a
one-step bound for $R_n$.

\begin{lemma}\label{lem:one-step}
For the unit vector $u_\theta=(\cos\theta,\sin\theta)^\transpose$,
\begin{equation*}
 \sum_{i=0}^2\frac1{\norm{K_i u_\theta}^2}
 =\frac{63}{65+16\cos(6\theta)}
 \le\frac97.
\end{equation*}
\end{lemma}

\begin{proof}
A unit eigenvector of $K_0$ for the eigenvalue $3$ is $e_1$, whose polar
angle is $0$.  Direct multiplication shows
\[
 K_1(1,\sqrt3)^\transpose=3(1,\sqrt3)^\transpose,
 \qquad
 K_2(1,-\sqrt3)^\transpose=3(1,-\sqrt3)^\transpose.
\]
Thus unit eigenvectors of $K_1$ and $K_2$ for the eigenvalue $3$ may be
chosen with polar angles $\pi/3$ and $-\pi/3$, respectively.  In each case,
the perpendicular unit vector is an eigenvector with eigenvalue $1$.  If
$\gamma$ denotes the chosen angle of the eigenvector for eigenvalue $3$,
decomposing $u_\theta$ in this orthonormal eigenbasis gives
\begin{align*}
\norm{K_i u_\theta}^2
&=9\cos^2(\theta-\gamma)+\sin^2(\theta-\gamma)\\
&=1+8\cos^2(\theta-\gamma)\\
&=5+4\cos(2\theta-2\gamma).
\end{align*}
Consequently,
\begin{align*}
\norm{K_0u_\theta}^2&=5+4\cos(2\theta),\\
\norm{K_1u_\theta}^2&=5+4\cos(2\theta-2\pi/3),\\
\norm{K_2u_\theta}^2&=5+4\cos(2\theta+2\pi/3).
\end{align*}
Put $x=2\theta$ and
\[
 d_j=5+4\cos\left(x+\frac{2\pi j}{3}\right),
 \qquad j=0,1,2.
\]
The three squared norms above are $d_0,d_1,d_2$ in some order.  The standard
three-angle identities give
\[
 d_0d_1+d_1d_2+d_2d_0=63,
 \qquad
 d_0d_1d_2=65+16\cos(3x).
\]
Therefore
\[
 \sum_{i=0}^2\frac1{\norm{K_i u_\theta}^2}
 =\sum_{j=0}^2\frac1{d_j}
 =\frac{63}{65+16\cos(6\theta)}
 \le\frac97.
\]
\end{proof}

\begin{lemma}\label{lem:universal-R}
For every $v\in\R^2\setminus\{0\}$ and every $n\ge1$,
\begin{equation}\label{eq:universal-R}
 R_n(v)\le\frac1{\norm{v}^2}\left(\frac97\right)^n.
\end{equation}
\end{lemma}

\begin{proof}
By homogeneity, Lemma~\ref{lem:one-step} gives, for every
$v\in\R^2\setminus\{0\}$,
\begin{equation}\label{eq:homogeneous-one-step}
 \sum_{i=0}^2\frac1{\norm{K_i v}^2}
 \le\frac9{7\norm{v}^2}.
\end{equation}
In particular, this proves \eqref{eq:universal-R} for $n=1$.  Applying
\eqref{eq:homogeneous-one-step} to $K_wv$ and summing over all words $w$ of
length $n$ gives
\[
 R_{n+1}(v)\le\frac97R_n(v).
\]
Induction proves \eqref{eq:universal-R}.
\end{proof}

\begin{theorem}\label{thm:general}
For any ordered pair of nonconstant harmonic functions $h_1,h_2$, the sequence
$S_{h_1,h_2}(n,p)$, $n\ge1$, is nondecreasing in $n$ for $p>1$.  For each
fixed pair, it is bounded uniformly in $n$ whenever
\[
 1<p<\frac{\log(35/3)}{\log 9}.
\]
In this range,
\[
 \nu_{h_1}\ll\nu_{h_2}
 \qquad\text{and}\qquad
 \frac{d\nu_{h_1}}{d\nu_{h_2}}\in L^p(\nu_{h_2}).
\]
\end{theorem}

\begin{proof}
Monotonicity follows directly from Proposition~\ref{prop:master-increment}.
Combining that proposition with Lemma~\ref{lem:universal-R} gives
\begin{align*}
0&\le S_{h_1,h_2}(n+1,p)-S_{h_1,h_2}(n,p)\\
&\le C_{p,h_1,h_2}
\left(\frac{9^p}{15}\right)^n
\frac1{\norm{z_{h_2}}^2}\left(\frac97\right)^n.
\end{align*}
The factor $1/\norm{z_{h_2}}^2$ is independent of $n$, so it may be absorbed
into $C_{p,h_1,h_2}$.
Thus
\[
0\le S_{h_1,h_2}(n+1,p)-S_{h_1,h_2}(n,p)
\le C_{p,h_1,h_2}
\left(\frac{3\cdot 9^p}{35}\right)^n.
\]
For $1<p<\log(35/3)/\log 9$, set
\[
 \rho=\frac{3\cdot9^p}{35},
\]
so $0<\rho<1$.  Lemma~\ref{lem:increment-summability}, applied with
$b_n=C_{p,h_1,h_2}\rho^n$, gives uniform boundedness and convergence.
The $L^p$ conclusion follows from
Proposition~\ref{prop:martingale-criterion}.
\end{proof}

\section{\texorpdfstring{The harmonic direction represented by
$(0,-1,1)$: lattice coding and the full interval}{The harmonic direction
represented by (0,-1,1): lattice coding and the full interval}}
\label{sec:boundary-direction}

Proposition~\ref{prop:master-increment} reduces the boundedness problem to
controlling the reciprocal sum
\[
 R_n(z_{h_2})=
 \sum_{\abs{w}=n}\frac1{\norm{K_wz_{h_2}}^2}.
\]
For a general denominator harmonic direction,
Section~\ref{sec:general-pair} bounded this sum by a geometric sequence.  The
harmonic direction represented by the boundary values $(0,-1,1)$ has
additional arithmetic structure: after a change of coordinates, every
product $K_we_2$ is represented by an integer lattice point.  Different
words of the same length give different lattice points, which allows us to
replace the word sum by a larger but elementary lattice sum.

To see why $e_2$ represents this harmonic direction, substitute
$(a,b,c)=(0,-1,1)$ into \eqref{eq:harmonic-coordinates}.  This gives
\[
 z_h=\begin{pmatrix}0\\\sqrt3\end{pmatrix}=\sqrt3e_2.
\]
Since
\[
 R_n(\lambda v)=\frac1{\lambda^2}R_n(v)
 \qquad(\lambda\in\R\setminus\{0\}),
\]
the fixed factor $\sqrt3$ has no effect on the growth rate.  It is therefore
enough to prove a bound for $R_n(e_2)$.

\subsection{The reciprocal lattice bound}

\begin{lemma}\label{lem:lattice}
There is an absolute constant $C>0$ such that
\[
 R_n(e_2)\le Cn
\]
for every $n\ge1$.
\end{lemma}

\begin{proof}
The proof has three parts: we encode the vectors $K_we_2$ by integer lattice
points, prove that this coding is injective at each fixed level, and then
estimate the resulting lattice sum by square shells.

For $q=(r,s)^\transpose\in\R^2$, write
\[
 Q(q)=r^2-rs+s^2.
\]
Introduce the invertible matrix
\[
 T=\begin{pmatrix}\sqrt3/2&0\\-1/2&1\end{pmatrix}
\]
and the integer matrices
\[
 L_0=\begin{pmatrix}3&0\\1&1\end{pmatrix},
 \qquad
 L_1=\begin{pmatrix}1&1\\0&3\end{pmatrix},
 \qquad
 L_2=\begin{pmatrix}2&-1\\-1&2\end{pmatrix}.
\]
Direct multiplication gives
\begin{equation}\label{eq:lattice-conjugacy}
 K_iT=TL_i
 \qquad(i=0,1,2),
\end{equation}
and
\begin{equation*}
 \norm{Tq}^2=Q(q)
 \qquad(q\in\R^2).
\end{equation*}

Define the lattice codes recursively by
\[
 q_\varnothing=(0,1)^\transpose,
 \qquad
 q_{wi}=L_iq_w.
\]
Thus, if $w=i_1\cdots i_n$, then
\[
 q_w=L_{i_n}\cdots L_{i_1}(0,1)^\transpose\in\Z^2.
\]
This recursion mirrors the convention $K_{wi}=K_iK_w$.  Since
$e_2=Tq_\varnothing$, repeated use of
\eqref{eq:lattice-conjugacy} gives
\begin{equation*}
 K_we_2=Tq_w.
\end{equation*}
Consequently,
\begin{equation}\label{eq:q-norm-form}
 \norm{K_we_2}^2=Q(q_w),
\end{equation}
and hence
\begin{equation}\label{eq:R-lattice-representation}
 R_n(e_2)=\sum_{\abs{w}=n}\frac1{Q(q_w)}.
\end{equation}

We next prove that, for each fixed $n$, the coding map
\[
 \{0,1,2\}^n\longrightarrow\Z^2,
 \qquad
 w\longmapsto q_w,
\]
is injective.  The key observation is that the residue of $q_w$ modulo $3$
records the last letter of $w$.

Indeed, for $q=(r,s)^\transpose$,
\begin{align*}
 L_0q&=(3r,r+s)^\transpose,\\
 L_1q&=(r+s,3s)^\transpose,\\
 L_2q&=(2r-s,-r+2s)^\transpose.
\end{align*}
The coordinate sums of these three vectors are, respectively,
$r+s+3r$, $r+s+3s$, and $r+s$.  Since the coordinate sum of
$q_\varnothing$ is $1$, induction gives
\begin{equation}\label{eq:coordinate-sum-congruence}
 r+s\equiv1\pmod3
 \qquad\text{whenever }q_w=(r,s)^\transpose.
\end{equation}
Under this congruence,
\begin{equation}\label{eq:last-letter-residues}
\begin{aligned}
 L_0(r,s)^\transpose&\equiv(0,1)^\transpose,\\
 L_1(r,s)^\transpose&\equiv(1,0)^\transpose,\\
 L_2(r,s)^\transpose&\equiv(2,2)^\transpose
\end{aligned}
\qquad\pmod3.
\end{equation}
For the last line, \eqref{eq:coordinate-sum-congruence} gives
$s\equiv1-r\pmod3$, and therefore
\[
 2r-s\equiv3r-1\equiv2,
 \qquad
 -r+2s\equiv2-3r\equiv2
 \pmod3.
\]
Thus the three residues in \eqref{eq:last-letter-residues} distinguish the
three possible last letters.

Now let $w=i_1\cdots i_n$ and $w'=j_1\cdots j_n$ satisfy $q_w=q_{w'}$.
Their common residue determines the last letter, so $i_n=j_n$.  Since each
$L_i$ has determinant $3$, it is invertible over $\R$, and we may cancel
this last matrix to obtain
\[
 q_{i_1\cdots i_{n-1}}=q_{j_1\cdots j_{n-1}}.
\]
Repeating the same argument recovers all the letters, from last to first.
Hence $w=w'$, proving that the coding map is injective at level $n$.

It remains to locate these lattice points and estimate the enlarged sum.
By the upper singular-value bound in \eqref{eq:K-basic},
\[
 \norm{K_we_2}\le3^n.
\]
It follows from \eqref{eq:q-norm-form} that
\begin{equation*}
 Q(q_w)\le9^n.
\end{equation*}
Moreover, $q_w\ne0$ because $q_\varnothing\ne0$ and all the
matrices $L_i$ are invertible.  Since the lattice codes at level $n$ are
distinct and all terms are positive, \eqref{eq:R-lattice-representation}
may be enlarged to
\begin{equation}\label{eq:enlarged-lattice-sum}
 R_n(e_2)
 \le
 \sum_{\substack{q\in\Z^2\setminus\{0\}\\Q(q)\le9^n}}
 \frac1{Q(q)}.
\end{equation}

Write $q=(r,s)^\transpose$ and set
$k=\max(\abs{r},\abs{s})$.  Then
\[
 Q(q)
 =\frac12(r^2+s^2)+\frac12(r-s)^2
 \ge\frac12(r^2+s^2)
 \ge\frac{k^2}{2}.
\]
Thus, if $q$ occurs in \eqref{eq:enlarged-lattice-sum}, then
\[
 k\le\sqrt2\,3^n,
 \qquad
 \frac1{Q(q)}\le\frac2{k^2}.
\]

For each $k\ge1$, the lattice points satisfying
$\max(\abs{r},\abs{s})=k$ form the boundary of the integer square
$[-k,k]^2$.  Their number is
\[
 (2k+1)^2-(2k-1)^2=8k.
\]
Therefore, with $N=\lceil\sqrt2\,3^n\rceil$,
\begin{align*}
 R_n(e_2)
 &\le\sum_{k=1}^{N}8k\frac2{k^2}\\
 &=16\sum_{k=1}^{N}\frac1k\\
 &\le16(1+\log N)
 \le Cn,
\end{align*}
where the last inequality follows from
$N\le2\sqrt2\,3^n$ for $n\ge1$.  This proves the lemma.
\end{proof}

Lemma~\ref{lem:lattice} replaces the geometric bound from
Section~\ref{sec:general-pair} by the much smaller estimate
$R_n(e_2)=O(n)$.  Proposition~\ref{prop:master-increment} then gives a
polynomial factor times $(9^p/15)^n$, which is summable on the full
conjectured interval.

\begin{corollary}
\label{cor:boundary-denominator}
Let $h_1,h_2$ be nonconstant harmonic functions, and assume that
$z_{h_2}$ is proportional to $e_2$, equivalently that the harmonic direction
of $h_2$ is represented by the boundary values $(0,-1,1)$.  Then, for every
fixed $p$ satisfying
\[
 1 < p<\frac{\log 15}{\log 9},
\]
one has
\[
 \sup_{n\ge1}S_{h_1,h_2}(n,p)<\infty.
\]
For $1<p<\log 15/\log 9$, it follows that
\[
 \nu_{h_1}\ll\nu_{h_2}
 \qquad\text{and}\qquad
 \frac{d\nu_{h_1}}{d\nu_{h_2}}\in L^p(\nu_{h_2}).
\]
\end{corollary}

\begin{proof}
Assume that
$1<p<\log 15/\log 9$.
The assumption on $h_2$ means that
\[
 z_{h_2}=\lambda e_2
\]
for some scalar $\lambda\in\R\setminus\{0\}$.  By homogeneity and
Lemma~\ref{lem:lattice},
\[
 R_n(z_{h_2})
 =\frac1{\lambda^2}R_n(e_2)
 \le\frac{C}{\lambda^2}n.
\]
The fixed factor $C/\lambda^2$ may be absorbed into a constant depending on
$h_2$.
Proposition~\ref{prop:master-increment} therefore yields
\[
0\le S_{h_1,h_2}(n+1,p)-S_{h_1,h_2}(n,p)
\le C_{p,h_1,h_2}n\left(\frac{9^p}{15}\right)^n.
\]
Set
\[
 \rho=\frac{9^p}{15},
 \qquad 0<\rho<1.
\]
Lemma~\ref{lem:increment-summability}, applied with
$b_n=C_{p,h_1,h_2}n\rho^n$, gives uniform boundedness and convergence.
The measure-theoretic conclusion follows from
Proposition~\ref{prop:martingale-criterion}.  This proves the corollary.
\end{proof}

We now specialize to the ordered pair of harmonic directions represented by
the boundary values $(1,0,0)$ and $(0,-1,1)$.  The corollary gives
boundedness below the endpoint, while the branch $0^n$, on which
$K_0e_1=3e_1$ and $K_0e_2=e_2$, gives the matching lower bounds.

\subsection{\texorpdfstring{The ordered pair with harmonic directions
represented by $(1,0,0)$ and $(0,-1,1)$}{The ordered pair represented by
(1,0,0) and (0,-1,1)}}

\begin{theorem}
\label{thm:sharp}
Let $h_1$ and $h_2$ have harmonic directions represented by boundary values
$(1,0,0)$ and $(0,-1,1)$, respectively.
For every $p>1$,
\[
 \sup_{n\ge1}S_{h_1,h_2}(n,p)<\infty
 \quad\Longleftrightarrow\quad
 p<\frac{\log 15}{\log 9}.
\]
Moreover, $\nu_{h_1}\ll\nu_{h_2}$, and, for every $p>1$,
\[
 \frac{d\nu_{h_1}}{d\nu_{h_2}}\in L^p(\nu_{h_2})
 \quad\Longleftrightarrow\quad
 p<\frac{\log 15}{\log 9}.
\]
The sequence $S_{h_1,h_2}(n,p)$, $n\ge1$, is strictly increasing for every
$p>1$.  At $p=\log 15/\log 9$ it grows at least linearly, and for
$p>\log 15/\log 9$,
\[
 S_{h_1,h_2}(n,p)\ge
 \left(\frac{9^p}{15}\right)^n.
\]
\end{theorem}

\begin{remark}
The quantity $S_{h_1,h_2}(n,p)$ is not symmetric in the two measures.  The
theorem makes no claim about the sharp threshold for the reversed ordered
pair.
\end{remark}

\begin{proof}
Under \eqref{eq:harmonic-coordinates}, the boundary representatives
$(1,0,0)$ and $(0,-1,1)$ correspond to $e_1$ and $\sqrt3e_2$,
respectively.  By the homogeneity of \eqref{eq:exact-cell-formula}, we use
$e_1$ and $e_2$ as representative vectors.
Boundedness below the endpoint follows from
Corollary~\ref{cor:boundary-denominator}.  It remains to prove strict
increase and the stated lower bounds.

The matrix $K_0$ acts particularly simply on these representative vectors:
\[
 K_0e_1=3e_1,
 \qquad
 K_0e_2=e_2.
\]
After $n$ repetitions,
\[
 K_{0^n}e_1=3^ne_1,
 \qquad
 K_{0^n}e_2=e_2.
\]
Substitution into \eqref{eq:exact-cell-formula} gives
\[
 \nu_{h_1}(F_0^n(K))=\frac{9^n}{15^n}=\left(\frac35\right)^n,
 \qquad
 \nu_{h_2}(F_0^n(K))=\frac1{15^n}=\left(\frac1{15}\right)^n.
\]
Hence the contribution of this word at level $n$ is
\begin{align*}
\nu_{h_1}(F_0^n(K))^p\nu_{h_2}(F_0^n(K))^{1-p}
&=\left[
 \left(\frac35\right)^p
 \left(\frac1{15}\right)^{1-p}
 \right]^n\\
&=\left(\frac{9^p}{15}\right)^n.
\end{align*}
This is one of the nonnegative terms in the level-$n$ sum defining
$S_{h_1,h_2}(n,p)$.  Therefore
\begin{equation}\label{eq:zero-branch-contribution}
 S_{h_1,h_2}(n,p)\ge\left(\frac{9^p}{15}\right)^n.
\end{equation}

We next prove that $S_{h_1,h_2}(n,p)$ is strictly increasing in $n$.  For the
level-$n$ word $0^n$, the vectors $K_{0^n}e_1$ and $K_{0^n}e_2$ remain
proportional to $e_1$ and $e_2$.
Consequently, the refinement proportions $P_i$ and $Q_i$ appearing in the
exact increment formula \eqref{eq:exact-increment} are obtained by applying
the three matrices $K_i$ directly to $e_1$ and $e_2$.
Since both vectors have norm $1$, the formulas from the proof of
Proposition~\ref{prop:master-increment} become
\[
 P_i=\frac{\norm{K_ie_1}^2}{15},
 \qquad
 Q_i=\frac{\norm{K_ie_2}^2}{15}.
\]
Direct multiplication gives
\[
\norm{K_0e_1}^2=9,
\quad
\norm{K_1e_1}^2=\norm{K_2e_1}^2=3,
\]
and
\[
\norm{K_0e_2}^2=1,
\quad
\norm{K_1e_2}^2=\norm{K_2e_2}^2=7.
\]
Hence the two refinement distributions are
\[
 (P_0,P_1,P_2)=\left(\frac35,\frac15,\frac15\right),
 \qquad
 (Q_0,Q_1,Q_2)=\left(\frac1{15},\frac7{15},\frac7{15}\right).
\]
They are different.  The strict form of Jensen's inequality established in
the proof of Lemma~\ref{lem:one-cell-comparison} therefore shows that
the corresponding bracket in \eqref{eq:exact-increment} is positive for
every $p>1$.  The sum inside that bracket equals
\begin{align*}
\sum_{i=0}^2P_i^pQ_i^{1-p}
&=\left(\frac35\right)^p\left(\frac1{15}\right)^{1-p}
 +2\left(\frac15\right)^p\left(\frac7{15}\right)^{1-p}\\
&=\frac{9^p}{15}
 +\frac{14}{15}\left(\frac37\right)^p.
\end{align*}

All summands in \eqref{eq:exact-increment} are nonnegative.  For the single
level-$n$ word $w=0^n$, the factor outside the bracket is the word
contribution $(9^p/15)^n$ calculated above.  Keeping only this word yields
\begin{align}
&S_{h_1,h_2}(n+1,p)-S_{h_1,h_2}(n,p)\notag\\
&\qquad\ge
\left[
 \frac{9^p}{15}
 +\frac{14}{15}\left(\frac37\right)^p-1
\right]
\left(\frac{9^p}{15}\right)^n.
\label{eq:zero-branch-increment}
\end{align}
Both factors on the right are positive for every $n\ge1$ and $p>1$.  Hence
$S_{h_1,h_2}(n,p)$ increases strictly at every level.

Set
\[
 p_*=\frac{\log 15}{\log 9},
 \qquad
 c_*=\frac{14}{15}\left(\frac37\right)^{p_*}>0.
\]
Since $9^{p_*}/15=1$, \eqref{eq:zero-branch-increment} gives
\[
 S_{h_1,h_2}(n+1,p_*)-S_{h_1,h_2}(n,p_*)\ge c_*.
\]
The second part of Lemma~\ref{lem:increment-summability} therefore yields
\[
 S_{h_1,h_2}(N,p_*)\ge S_{h_1,h_2}(1,p_*)+c_*(N-1),
 \qquad N\ge1,
\]
which proves linear growth at the endpoint.

Together with \eqref{eq:zero-branch-contribution}, this proves the
boundedness threshold.  Choose $p_0$ so that
\[
 1<p_0<\frac{\log 15}{\log 9}.
\]
Corollary~\ref{cor:boundary-denominator} then gives
$\nu_{h_1}\ll\nu_{h_2}$.  For each fixed $p>1$,
Proposition~\ref{prop:martingale-criterion} identifies membership of the
Radon--Nikodym density in $L^p(\nu_{h_2})$ with boundedness of
$S_{h_1,h_2}(n,p)$.  The boundedness threshold just proved therefore gives
the stated $L^p$ equivalence.
\end{proof}

\section*{Acknowledgment}

This work was developed with assistance from \emph{GPT-5.6}, which was used
to explore ideas and support aspects of the technical development and
writing.  All results and arguments were independently verified by the
author, who assumes full responsibility for the content of the paper.

\end{document}